\documentclass{scrartcl} 

\usepackage{color}
\usepackage{amsmath,amsfonts,amssymb,amsthm}
\usepackage{graphicx}

\newtheorem{theorem}{Theorem}[section]%
\newtheorem{corollary}[theorem]{Corollary}%
\newtheorem{lemma}[theorem]{Lemma}%
\newtheorem{proposition}[theorem]{Proposition}%
\newtheorem{remark}[theorem]{Remark}%

\title{Absolute Continuity of Complex Martingales and of Solutions to Complex Smoothing Equations } 

\author{%
  Ewa~Damek\footnote{University of Wroc\l aw, Poland. E-mail: \texttt{edamek@math.uni.wroc.pl}}
 \and 
  Sebastian~Mentemeier\footnote{University of Kassel,
    Germany.   E-mail: \texttt{mentemeier@mathematik.uni-kassel.de} }}

\begin{document}

\maketitle

\begin{abstract}
Let $X$ be a $\mathbb{C}$-valued random variable with the property that
$$X \ \text{ has the same law as }\ \sum_{j\ge1} T_j X_j$$
where $X_j$ are i.i.d.\ copies of $X$, which are independent of the (given) $\mathbb{C}$-valued random variables $ (T_j)_{j\ge1}$. We provide a simple criterion for the absolute continuity of the law of $X$ that requires, besides the known conditions for the existence of $X$, only finiteness of the first and second moment of  $N$ - the number of nonzero weights $T_j$.
Our criterion applies in particular to Biggins' martingale with complex parameter.  
\end{abstract}


\newcommand{\red}{\color{red}}
\newcommand{\blue}{}
\newcommand{\R}{\mathbb{R}}
\renewcommand{\P}{\mathbb{P}}
\newcommand{\E}{\mathbb{E}}
\newcommand{\C}{\mathbb{C}}
\newcommand{\N}{\mathbb{N}}
\newcommand{\V}{\mathbb{V}}
\newcommand{\eqdist}{\stackrel{\mathrm{law}}{=}}
\renewcommand{\c}{\mathcal}

\newcommand{\eps}{\varepsilon}
\newcommand{\Ind}{\mathbf{1}}
\newcommand{\8}{\infty}
\newcommand{\s}{\sigma}

\newcommand{\imag}{\mathrm{i}}

\newcommand{\Sp}{\mathbb{S}_\ge}
\newcommand{\Sd}{S^{d-1}}

\newcommand{\diag}{\mathrm{diag}}
\newcommand{\innerprod}[1]{\langle #1 \rangle}

\newcommand{\vect}[2]{\left( \begin{array}{c} #1 \\ #2 \end{array} \right)}
\newcommand{\diagmatr}[2]{\left( \begin{array}{cc} #1 & 0  \\ 0 & #2  \end{array} \right)}






\section{Introduction}

In a variety of models coming from theoretical computer science, applied probability, economics or statistical physics, quantities of interest exhibit asymptotic fluctuations that do not have a normal or $\alpha$-stable distribution. In many cases, the limiting law $\mu$ can be characterized as a fixed point of a mapping $\c S$ of the form
\begin{equation}\label{eq:ST} \c S(\mu) ~=~ \text{ Law of } \Big(\sum_{j \ge 1} T_j X_j \Big), \end{equation}
where $X_j$ are i.i.d.\ complex-valued random variables with law $\mu$ and independent of the given  complex variables $(T_j)_{j\ge1}$. See \cite{Meiners.Mentemeier:2017} and references therein for a list of examples.

The fixed point property $\mu=\c S (\mu)$ then {\blue may} and shall be used to analyze properties of $\mu$. 
Let us stress at this early point that $\c S$ usually has multiple fixed points, which have to be analyzed by different methods. They can roughly be classified by a parameter $\alpha$: The first class of fixed points are mixtures of $\alpha$-stable laws, while the second class of fixed points appears only for $\alpha \ge1$. Fixed points of the second class are limit of martingales in an associated weighted branching process. Under an additional very mild assumption, fixed points from the second class are integrable.

In this note, we will study absolute continuity of fixed points of $\c S$ and take advantage of the classification described above, which was recently given in \cite{Meiners.Mentemeier:2017}. {\blue This simplifies essentially the approach.} For fixed points from the first class, absolute continuity can be proved along similar lines as for infinitely divisble laws. For fixed points from the second class, 
{\blue we apply Fourier analytic methods 
and then integrability allows us to work with derivatives of the characteristic function}. 

Our setting includes as well the case of real-valued $X$ and $(T_j)_{j\ge1}$. For the case of nonnegative $X$ and $(T_j)_{j\ge1}$, general results have been obtained in \cite{Biggins.Grey:1979} and \cite{Liu:2001}. The real- and complex-valued setup has been treated recently also in \cite{Leckey:2016}. The paper \cite{Leckey:2016} covers more general classes of fixed point equations, but at the price of stronger assumptions than imposed here. {\blue For instance  negative moments of $T_j$ are required which is not natural for \eqref{eq:ST}, while integrability of $T_j$ is.} The approach in \cite{Leckey:2016} is different, for it does not take into account a-priori knowledge as the classification of fixed points described above. Our approach is based on ideas in \cite{Biggins.Grey:1979}. Absolute continuity of a specific complex-valued model is also studied in \cite{Chauvin+Liu+Pouyanne:2014}.

We continue in Section \ref{sect:Results} with a precise description of the setup and the set of fixed points  of $\c S$. Then we state our results and describe several examples that motivated our study. The  proofs are given in Section \ref{sect:Proofs}.

\section{Statement of Results} \label{sect:Results}

 \subsection{Solutions to complex smoothing equations}

Let $(T_j)_{j\ge1}$ be complex-valued random variables, satisfying
$$ N:=\#\{j \, : \, T_j \neq 0\} ~=~ \max \{ j \, : \, T_j \neq 0\} < \infty \quad \P\text{-a.s.}$$
Let $X$ be a complex random variable with law $\mu$ such that $\c S(\mu)=\mu$. Then 
\begin{equation}\label{eq:SFPE} X~\eqdist~\sum_{j=1}^N T_j X_j. \end{equation}
This gives rise as well to an equation for the characteristic function $\phi(\xi)=\E \big[ e^{-\imag \innerprod{\xi,X}}\big]$\footnote{Note that in the definition of the characteristic function, the identification $\C\simeq\R^2$ and the real inner product is used.}, namely
\begin{equation} 
\label{eq:SFPE:characteristic}
\phi(\xi) ~=~ \E \Big[ \prod_{j=1}^N \phi( \bar{T_j} \xi) \Big].
\end{equation}

The set of all solutions to $\c S(\mu)=\mu$ has been described in \cite{Meiners.Mentemeier:2017} under the following mild assumptions. Upon introducing the function
$$ m(s)~:=~\E \Big[ \sum_{j \ge 1} |T_j|^s \Big],$$
we assume that
\begin{equation}\label{a1}\tag{A1} m(0)= \E[N] >1. \end{equation}
\begin{equation}\label{a2}\tag{A2} m(\alpha)=1 \text{ for some } \alpha >0. \end{equation}
Under \eqref{a2}, $W_1:= \sum_{j=1}^N |T_j|^\alpha$ defines a mean one random variable. Assume further
\begin{equation}\label{a3}\tag{A3}  m'(\alpha):= \E \Big[ \sum_{j=1}^N |T_j|^\alpha \log |T_j| \Big] \in (-\infty,0) \ \text{ and } \ \E \Big[ W_1 \log_+ W_1 \Big] < \infty. \end{equation}
Let $U \subset \C$ be the smallest closed multiplicative subgroup generated by the support of $(T_j)_{j\ge1}$.

Suppose that \eqref{a1}--\eqref{a3} hold with $\alpha \neq 1$ (in the case $\alpha=1$, an additional technical assumption is required). Then, by \cite[Theorem 1.2]{Meiners.Mentemeier:2017}, there exists a nonnegative random variable $W$ with unit mean and a $\C$-valued random variable $Z$ such that if the law of $X$ is a fixed point of $\c S$, then
\begin{equation}\label{eq:decompositionX} X \eqdist Y_W + xZ. \end{equation}
where $x \in \C$ and $(Y_t)_{t \ge 0}$ is a complex-valued L\'evy process with the invariance property
\begin{equation} \label{eq:operator stable} uY_t \eqdist Y_{|u|^\alpha t} \qquad \text{ for all } u \in U, t >0,  \end{equation} and $(Y_t)_{t \ge 0}$ is  independent of $(W,Z)$. Note that $Y_t \equiv 0$ is a valid choice.
If $(Y_t)_{t \ge 0}$ is nontrivial, it holds $\E\big[ |Y_W|^\alpha \big] = \infty$, see \cite[Remark 1.4]{Meiners.Mentemeier:2017}

\subsubsection{Martingales and the Weighted Branching Process}

To give a description of $W$ and $Z$, let us define a weighted branching process as follows: Let $\V= \bigcup_{n=0}^\infty \N^n$ denote the infinite tree with Harris-Ulam labelling and root $\emptyset$. For each $v \in \V$, we denote by $|v|$ its generation. 
To each $v \in \V$, we attach an independent copy $(T_1(v), T_2(v), \dots)$ of $(T_j)_{j\ge1}$ and define the weighted branching process by
$$ L(\emptyset) := 1, \quad L(vi) := T_i(v) L(v),$$
where $vi$ denotes concatenation: if $v=v_1 \cdots v_k$, then $vi=v_1 \cdots v_k i$.

Then $W := \lim_{n \to \infty} \sum_{|v|=n} |L(v)|^\alpha$ with $\E[W]=1$. Here, \eqref{a2} implies that $W_n$ is a martingale and \eqref{a3} guarantees its convergence in $L^1$ by Biggins' theorem.

$Z=0$ unless $\E \big[ \sum_{j=1}^N T_j \big]=1$ and $\alpha \ge 1$.  If these requirements are satisfied, then $Z_n := \sum_{|v|=n} L(v)$ defines a $\C$-valued martingale with mean one.
In our results, we will require that
\begin{equation}\label{eq:Z}\tag{Z1} \lim_{n \to \infty} Z_n \text{ exists a.s.\ and in $L^1$} \end{equation}
        We have $Z:=\lim_{n \to \infty} Z_n$, if \eqref{eq:Z} holds, and $Z:=0$ otherwise. 
Under \eqref{a1}-\eqref{a2}, 
 a sufficient condition for $Z_n$ to converge a.s. and in $L^p$ for all $p < \alpha$ is $\alpha \in (1,2)$ and
\begin{equation}\label{a4}\tag{A4}
m'(\alpha) \le 0 \ \text{ and } \ \E \Big[ \big|Z_1\big|^\alpha \log^{2+\epsilon}_+ |Z_1| \Big]  \text{ for some } \epsilon >0,
\end{equation}
see \cite[Theorem 2.1]{Kolesko.Meiners:2017}. Hence, under \eqref{a1}, \eqref{a2} and \eqref{a4}, $\E[Z]=1$ and $Z$ is in $L^p$ for all $p < \alpha$. 

If $\alpha=2$, under mild conditions, $Z$ is either 0 or a constant a.s., see \cite[Proposition 2.2]{Kolesko.Meiners:2017}, \cite[Proposition 1.1.]{Meiners.Mentemeier:2017}. If $\alpha=1$, then $\E [\sum_{j=1}^N |T_j|] = \E [\sum_{j=1}^N T_j]=1$ implies $U \subset \R_+$, hence $Z=W$. Continuity properties of the nonnegative random variable $W$ {\blue have been} studied in \cite{Biggins.Grey:1979, Liu:2001}.


\subsection{Results}

We study the absolute continuity of $Z$. By the discussion above, we may focus on the case $1  <\alpha <2$. We further assume that $\P(N=0)=0$, since otherwise all solutions have an atom in zero. 
%
%

\begin{theorem}\label{thm:continuity of Z}
Suppose $N>0$ a.s., \eqref{a1}-\eqref{a2} with $\alpha \in (1,2)$, \eqref{eq:Z}  and $\mathrm{supp}(Z) \nsubseteq \R$ together with
\begin{equation}\tag{C1} \label{c1} \quad \E\big[ N^2 \big] < \infty \ \text{ and } \ \E \Big[ N \sum_{j=1}^N \log_+ |T_j| \Big] < \infty. \end{equation}
Then the law of $Z$ is absolutely continuous. \\ If $\mathrm{supp}(Z) \subset \R$, then \eqref{c1} can be replaced with the assumption $\E [N] < \infty$. {\blue This applies in particular to \eqref{eq:ST} with real valued $T_j$}.
\end{theorem}

As mentioned before, \eqref{a4} is a mild sufficient condition for \eqref{eq:Z}. 
If higher order moment conditions on $Z$ and $N$ are satisfied, one can prove further smoothness properties of the Fourier transform of $Z$, see Remark \ref{rem:m2 finite}.

Concerning $(Y_t)_{t \ge 0}$, standard arguments yield the following continuity result:

\begin{proposition}\label{lem:Yt}
Suppose \eqref{a1}-\eqref{a2} with $\alpha \in (0,2]$. Suppose that $(Y_t)_{t \ge 0}$ is a nondegenerate complex-valued L\'evy process satisfying \eqref{eq:operator stable} and that there is no $U$-invariant $\R$-linear subspace of $\C$.  Then for each $t>0$, the law of $Y_t$ is absolutely continuous.
\end{proposition}

Combining both results and using that $(Y_t)_{t \ge 0}$ is independent of $(W,Z)$ in the representation \eqref{eq:decompositionX}, we have:

 \begin{corollary}\label{thm:continuity}
Suppose \eqref{a1}-\eqref{a4}, (C1) and that 
there is no $U$-invariant $\R$-linear subspace of $\C$. Then the law of any nontrivial solution to \eqref{eq:SFPE} is absolutely continouous.
\end{corollary}

\subsection{Examples}

\subsubsection*{Biggins' martingale with complex parameter}

A branching random walk is defined as follows. An ancestor at the origin produces offspring which is displaced on $\R$ according to a point process. Each new particle then produces again offspring independently of all other particles according to the same law. Denote the positions of the $n$-th generation particles by $(S(v))_{|v|=n}$ and suppose that for some $\lambda \in \C$,
$$ \mathfrak{m}(\lambda) ~:=~ \E \Big[ \sum_{|v|=1} e^{-\lambda S(v)}\Big]$$
exists and is nonzero. Then
$$ \c W_n(\lambda) ~:=~ \mathfrak{m}(\lambda)^{-n} \sum_{|v|=n} e^{-\lambda S(v)}$$
defines a $\C$-valued martingale that coincides with $Z_n$ upon identifying $$T_j=\mathfrak{m}(\lambda)^{-1}e^{-\lambda S(j)}.$$ These complex martingales were studied in \cite{Biggins:1992} to analyze the frequencies of particles with a certain speed in the branching random walk.

Let us consider a simple branching random walk with binary branching, i.e.,  $S(1), S(2)$ are i.i.d.\ with $\P(S(1)=1)=\P(S(1)=-1)=1/2$. Then $\mathfrak{m}(\lambda)=2 \cosh(\lambda)$,
$$ T_j=\frac{e^{-\lambda S(j)}}{2 \cosh(\lambda)}, \quad j=1,2; \qquad \qquad m(s) ~=~ \frac{2}{2^s} \frac{\cosh(s\Re(\lambda))}{|\cosh(\lambda)|^s}.$$
For given values of $\lambda$, the assumptions of Theorem \ref{thm:continuity of Z} are readily checked. 
Figures \ref{fig:1},\ref{fig:2} show estimates of the density of $W$ for different values of $\lambda$, based on the simulation algorithm proposed in \cite{Chen.Olvera-Cravioto:2015}. Sample size $n=10^6$, $10^2$ simulation steps. \nocite{R}


\begin{figure}
\centering
\includegraphics[width=.7\textwidth]{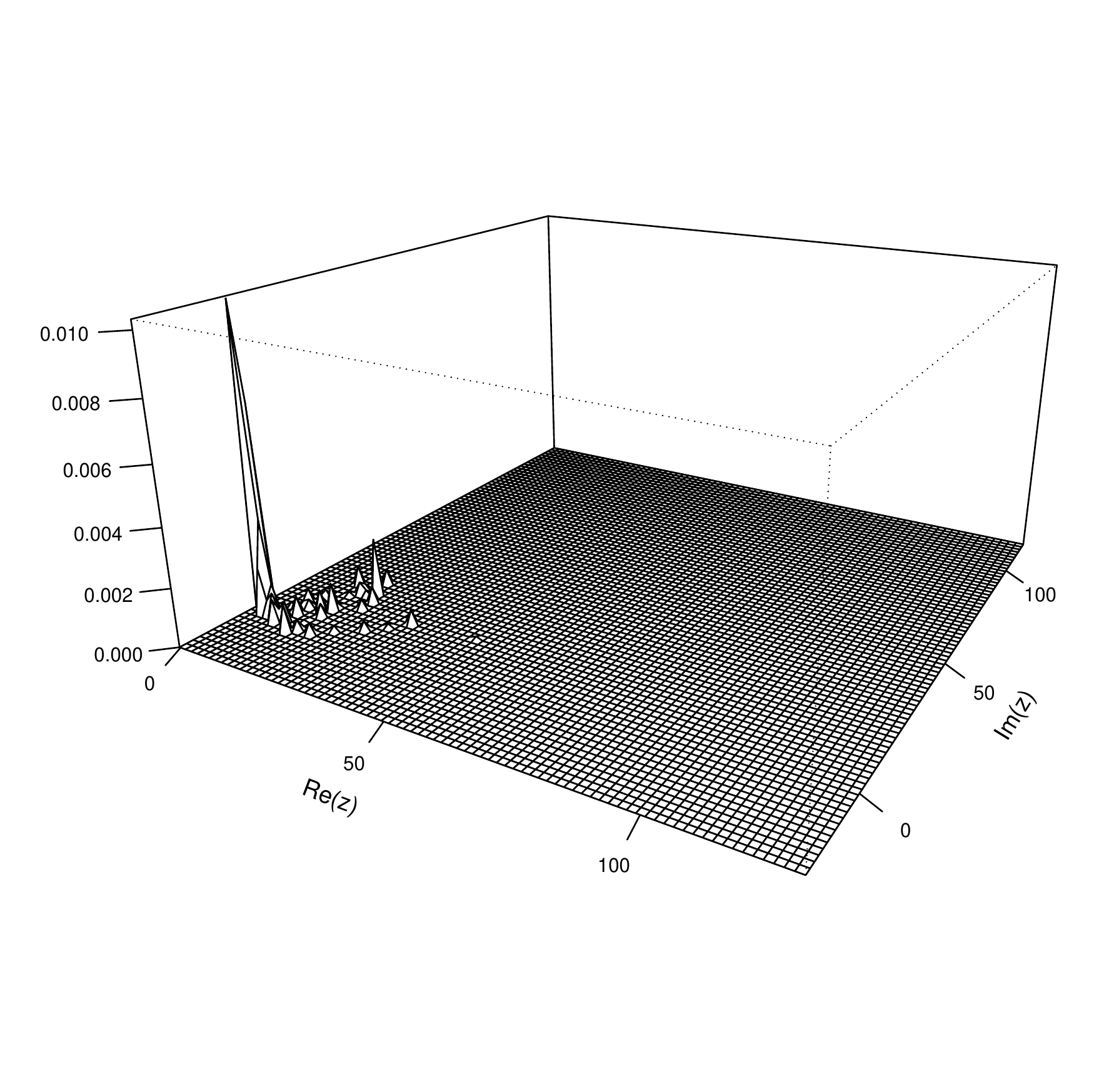}
\caption{Density estimate for Biggins' martingale with $\lambda=2.15*\exp(2 \pi \imag/23)$}
\label{fig:1}
\end{figure}

\begin{figure}
\centering
\includegraphics[width=.7\textwidth]{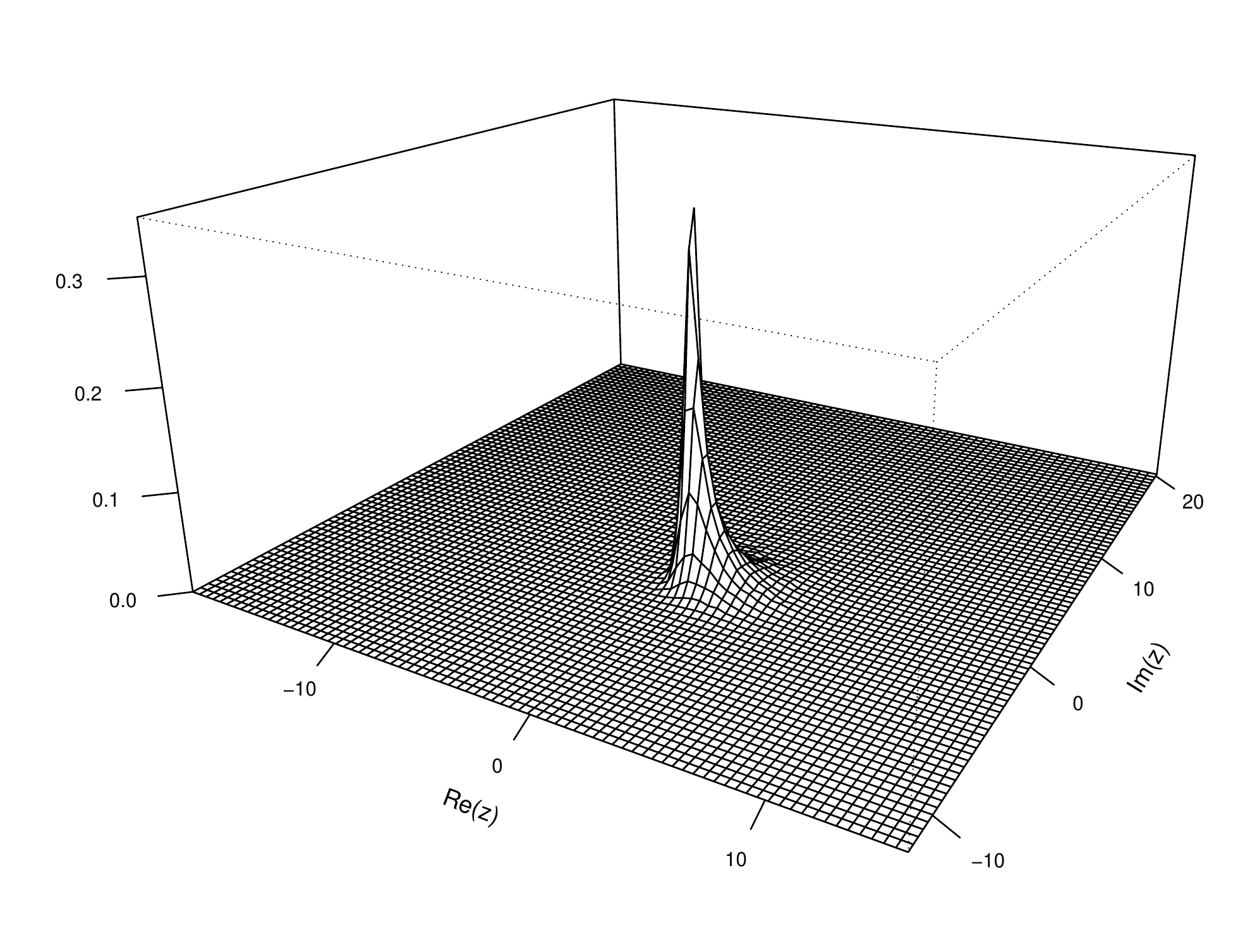}
\caption{Density estimate for Biggins' martingale with $\lambda=\exp( \pi \imag/4)$}
\label{fig:2}
\end{figure}


\subsubsection*{Cyclic P\'olya urns}

A cyclic P\'olya urn consists of balls of $b$ different types. Each time a ball of type $m$ is drawn, it is placed back into the urn together with a ball of type $m+1 \mod b$. If $b\ge 7$,  the asymptotic fluctuations of the proportion of balls of a given type are described in terms of a complex random variable $X$ with finite variance that satisfies 
$$ X \eqdist U^\zeta X_1 + \zeta (1-U)^\zeta X_2,$$
where $\zeta=\omega_b$ and $X_1, X_2$ are i.i.d.\ copies of $X$ which are independent of $U$, which is a uniform $[0,1]$-random variable; see e.g. \cite{Knape+Neininger:2014}.

We show how our result applies.
Assumptions \eqref{a1}--\eqref{a4} and \eqref{eq:Z} are readily checked, $\alpha=1/\Re(\zeta) \in (1,2)$ as soon as $b \ge 7$. Since the solution of interest has a second moment, it has to be $X=xZ$ for some $x \in \C$.  The set $\c Z :=\mathrm{supp}(Z)$ has to satisfy 
$$ u^\zeta \c Z + \zeta (1-u)^\zeta \c Z ~ \subset ~ \c Z \qquad \text{ for all } u \in [0,1]$$
which yields that $\c Z \nsubseteq \R$. Hence Theorem \ref{thm:continuity of Z} applies and shows that $X$ has a density. Figure \ref{fig:3} shows estimates of the density for different values of $b$, again based on the simulation algorithm proposed in \cite{Chen.Olvera-Cravioto:2015}. Sample size $n=10^5$, $10^2$ simulation steps.

\begin{figure}
\centering
\includegraphics[width=.96\textwidth]{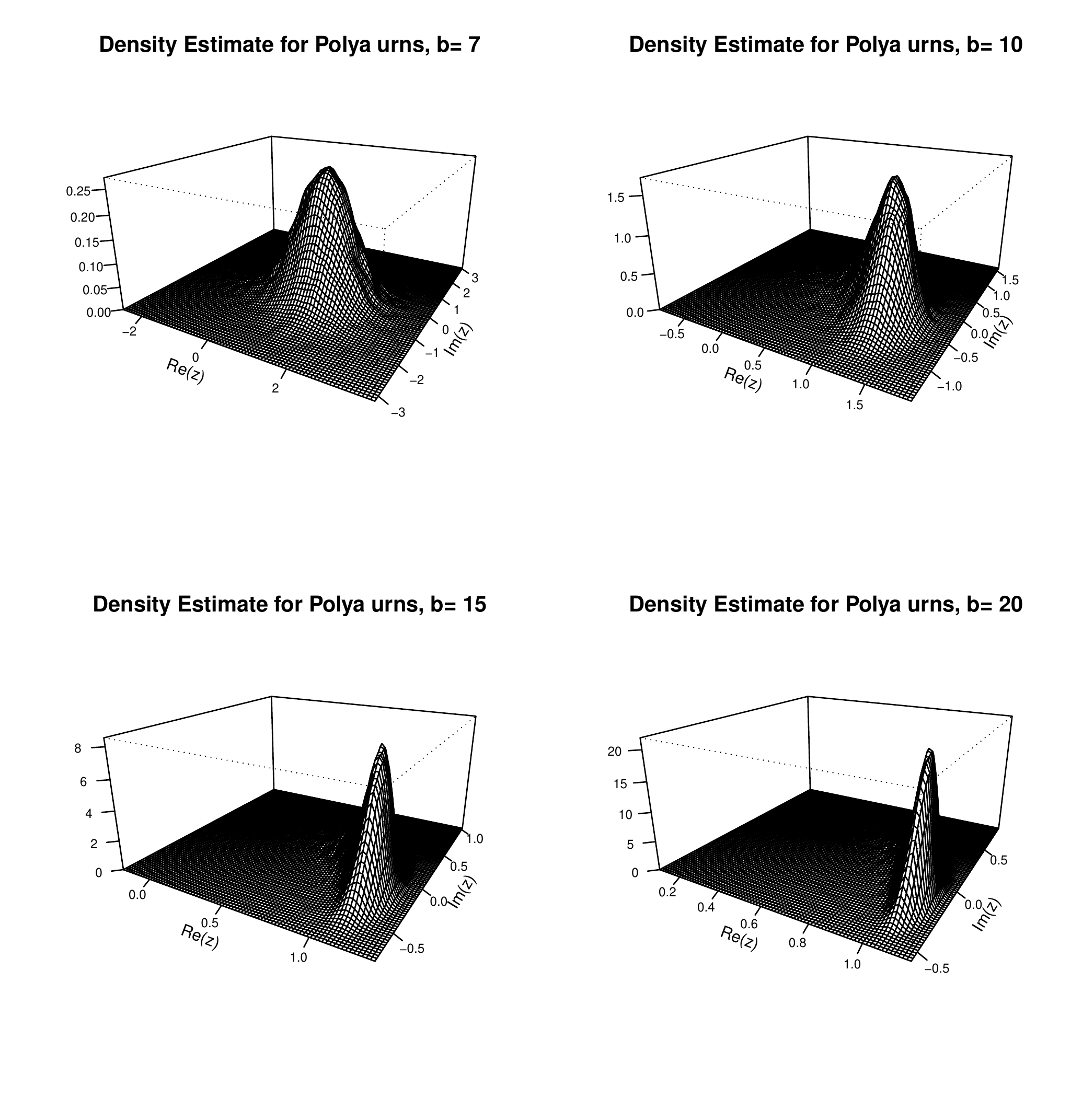}
\caption{Density estimates for cyclic P\'olya urns}
\label{fig:3}
\end{figure}

\section{Proofs} \label{sect:Proofs}

We start with the short proof of Proposition \ref{lem:Yt}.

Let $X$ be a random vector in $\R^d$ with characteristic function $\phi$. Then (the law of) $X$ is called {\em full}, if for all $v \neq 0$ in $\R^d$, $\innerprod{v,X}$ is not a point mass.  A complex-valued random variable $X$ is full, if it is full upon identifying $\C \simeq \R^2$. If $X$ is full, then there is $\epsilon >0$ such that $|\phi(\xi)|<1$ for all $0 < |\xi| < \epsilon$, see \cite[Lemma 1.3.15]{Meerschaert+Scheffler:2001}.

\begin{proof}[Proof of Proposition \ref{lem:Yt}.]
If there is no $U$-invariant linear subspace, then the invariance property \eqref{eq:operator stable} yields that the support of $Y_t$ is also not contained in a proper linear subspace of $C$, hence $Y_t$ is full. By \eqref{a1} and \eqref{a2}, the function $m$ is not constant, hence there is $u \in U$ with $|u| \neq 1$. Then, using that $Y_t$ is infinitely divisible, Eq. \eqref{eq:operator stable} yields that $Y_t$ is {\em operator semistable} (see \cite[Definition 7.1.2]{Meerschaert+Scheffler:2001}). By \cite[Theorem 7.1.15]{Meerschaert+Scheffler:2001}, a full operator semistable law has a density with respect to Lebesgue measure.
\end{proof}

%


\subsection{Proof of Theorem  \ref{thm:continuity of Z}}

\begin{lemma}\label{lem:support multiplicative}
Assume \eqref{a1}, \eqref{a2} and \eqref{eq:Z}. Then $\mathrm{supp}(Z)$ is closed under multiplication.
\end{lemma}

\begin{proof}
Up to obvious modifications, this can be proved along the same lines as Thm. 2 in \cite{Biggins.Grey:1979}. 
\end{proof}

In the following, we restrict our attention to the case where $Z$ is properly $\C$-valued, i.e., $\mathrm{supp}(Z) \nsubseteq \R$. The simpler case $\mathrm{supp}(Z) \subset \R$ requires only minor modifications.

If $\mathrm{supp}(Z) \nsubseteq \R$, then Lemma \ref{lem:support multiplicative} yields that $\mathrm{supp}(Z)$ is not contained in any affine $\R$-linear subspace of $\C$, hence $Z$ is full.

\begin{lemma} \label{lem:decay of phi} Assume \eqref{a1}, \eqref{a2} and \eqref{eq:Z}, as well as $N\ge 1$ a.s. and $\E[N]< \infty$. Then
$\ell~:=~\limsup_{|\xi| \to \infty} |\phi(\xi)|=0$. 
\end{lemma}

\begin{proof} 
By the same arguments as in \cite[Lemma 3.1\,(i)]{Liu:2001}, $\ell \in \{0,1\}$. 
As the next step, we prove that $|\phi(\xi)|<1$ for all $\xi \neq 0$. Since $Z$ is full, \cite[Lemma 1.3.15]{Meerschaert+Scheffler:2001} yields that there is $\eta >0$ such that $|\phi(\xi)|<1$ for all $0 < |\xi| < \eta$. Suppose 
$$ R:= \inf\big\{ r >0 \, : \, \exists\,\xi \text{ with } |\xi|=r \text{ s.t. } |\phi(\xi)|=1 \big\} < \infty.$$
Then choose $\xi^*$ with $|\xi^*|=R$ and $|\phi(\xi^*)|=1$. Taking absolute values on both sides of  Eq. \eqref{eq:SFPE:characteristic} yields $1 \le |\phi(\xi^*)| \le \E |\phi(\bar{T}_1 \xi^*)| \le 1,$ thus $|\phi(\bar{T}_1\xi^*)|=1$ a.s. But this contradicts $\P(0 < |T_1| < 1)>0$, which follows from $\E N>1$ and $m(\alpha)=1$. Hence $R=\infty$.

It remains to prove $\ell<1$. By \eqref{a2} and the branching property, $\E \sum_{|v|=n} |L(v)|^\alpha =1$ for all $n \in \N$, which yields that the expected number of summands exceeding 1 has to be smaller than one.   In addition, $\E[ \# \{v \, : \, |v|=n\}] = {\blue (\E[N])^n} < \infty$ gives that we can choose $\delta$ and $n$ such that
$$ \c L ~:=~ \{v \, : \, |v|=n, \ \delta \le |L(v)| \le 1\}$$
satisfies $1 < \E[\# \c L] < \infty$. Hence, for the moment generating function $\kappa(s):=\E \big[s^{\# \c L}\big]$ it holds $s-\kappa(s)>0$ for all $s \in (\eta,1)$, where $\eta$ is the unique root of $\kappa(s)-s=0$ on the interval $[0,1)$.

    
            Suppose $\ell=1.$ By the previous step, for sufficiently small $\epsilon>0$, there are $0<t_1<t_2$ with $t_1 < \delta t_2$ s.t.\ $|\phi(\xi)|<1-\epsilon$ for all $t_1 < |\xi| < t_2$, while there is $\xi^*$ with $|\xi^*|=t_2$ s.t.\ $|\phi(\xi^*)|=1-\epsilon$. By iterating Eq. \eqref{eq:SFPE:characteristic}, we obtain
$$ 1-\epsilon = |\phi(\xi^*)| ~\le~ \E\Big[ \prod_{v \in \c L} \big| \phi\big(\overline{L(v)} \xi^*\big)\big| \Big] \le \E \big[ (1-\epsilon)^{\# L}\big] ~=~ \kappa(1-\epsilon),$$
which contradicts $s > \kappa(s)$ for all $s \in (\eta, 1)$.
\end{proof}


%

\subsubsection*{Derivatives of the characteristic function}

To proceed further, we will consider the complex derivatives  $\partial_{\bar \xi} \phi(\xi)$ and $\partial_\xi \phi(\xi)$. Note that $\phi$ is differentiable as soon as $\E\big[ |Z| \big] < \infty$. 

One has {\blue to be careful}, because in the definition of $\phi$, the identification $\C=\R^2$ and the real inner product is used. We write $\xi=\xi _1+i\xi _2$ and
\begin{equation*}
\partial_{\xi}=\frac{1}{2}(\partial_{\xi _1}-i\partial_{\xi _2}),\quad
\partial_{\bar \xi}=\frac{1}{2}(\partial_{\xi _1}+i\partial_{\xi _2}) 
\end{equation*}
The characteristic function $\phi$ is given by
$$ \phi(\xi) ~=~ \E \Big[ \exp \big(-\imag \frac12 (\xi \bar Z +\bar \xi  Z\big) \big)\Big]=\E \Big[ \exp \big(-\imag \frac12 (\xi _1X + \xi _2 Y\big) \big)\Big],$$
where $Z=X+iY$. Hence
$$ \partial_{\xi} \phi(\xi) ~=~ \E \Big[ -\frac{\imag}{2} \bar Z \exp \big(-\imag \frac12 (\xi \bar Z + \bar \xi  Z\big) \big)\Big],$$
$$ \partial_{\bar{\xi}} \phi(\xi) ~=~ \E \Big[-\frac{ \imag }{2}Z\exp \big(-\imag \frac12 (\xi \bar Z +\bar \xi Z\big) \big)\Big].$$
because $\partial_\xi (\xi z) = z$, $\partial_{\bar \xi} (\xi z) =0$ for $z\in \C$.
Therefore, by the chain rule for complex differentiation (using Wirtinger derivatives)
\begin{equation}\label{eq:derivatives of phi} \partial_{\xi} \phi(\bar T \xi) = \bar{T} (\partial_\xi \phi)(\bar T \xi), \qquad \partial_{\bar \xi} \phi(\bar T \xi) = T (\partial_{\bar \xi} \phi)(\bar T \xi). \end{equation}
As the first step, we are going to prove decay rates for both derivatives.

\begin{lemma}\label{lem:decay rate g}
Suppose $N>0$ a.s., \eqref{a1}-\eqref{a2} with $\alpha \in (1,2)$, \eqref{eq:Z} and $\E[N]<\infty$. Then there is a finite constant $C$ such that 
\begin{equation}\label{eq:decay lemma} | \partial_{\xi} \phi(\xi )|\leq  C(1+|\xi |)^{-1} \quad \text{ and } \quad
| \partial_{\bar \xi} \phi(\xi )|\leq  C(1+|\xi |)^{-1} \quad \text{ for all } \xi \in \C. \end{equation}
\end{lemma}
\begin{remark}
{\blue If $\mathrm{supp} (Z)\subset \R $ it follows that $ \partial_{\xi} \phi(\xi ), \partial_{\bar \xi} \phi(\xi )$ are square integrable.}
\end{remark}
\begin{proof} We will prove the estimate for $\partial_{\bar \xi} \phi$. The proof for $\partial_{\xi } \phi$ is completely analogous, up to replacing $T_j$ by $\bar{T}_j$. Define $g(\xi ):=\partial _{\bar \xi}\phi (\xi )$. 
Then, differentiating both sides of Eq. \eqref{eq:SFPE:characteristic} and using \eqref{eq:derivatives of phi}
\begin{equation}\label{eq:functional equation for g}
g(\xi )=\E \Big[\sum_{j=1}^N T_j g(\bar T_j\xi ) \prod_{i\neq j} \phi (\bar T_i\xi ) \Big].
\end{equation}
Note that the right hand side is finite by using that $m(1)<\infty$ and that $g$ is bounded by $\E|Z|< \infty$. 
By Lemma \ref{lem:decay of phi}, for every $\eps $, there is $t_{\eps }$ such that
$|\phi (\xi )|<\eps $  for every $| \xi |>t_{\eps }$. Given $\delta >0$ let
$$
N_{\delta }=\sum _{j=1}^N\Ind \{|T_j|>\delta \}.$$ 
If $N_{\delta }\geq 1$ and $|\xi |> t_{\eps }\delta ^{-1}$ then for all $1 \le j \le N$,
\begin{equation}\label{eq:estimate epsilon Ndelta}
\prod_{i \neq j} |\phi (T_i\xi )|\leq \eps ^{N_{\delta }-1} \end{equation}
and hence 
\begin{equation}\label{eq:estimate g}
|g(\xi )|\leq \E \Big[ \varepsilon^{(N_\delta -1)_+}  \sum_{j=1}^N |T_j||g(\bar T_j\xi )|   \Big] \qquad \text{ for } |\xi| > t_\varepsilon \delta^{-1}.
\end{equation}
Define a complex valued random variable $B$ by
\begin{equation}
\E h(B )=q_{\eps, \delta }^{-1} \E \Big[ \varepsilon^{(N_\delta -1)_+}  \sum_{j=1}^N |T_j| h(\bar T_j) \Big],
\end{equation}
where $q_{\eps, \delta }= \E \Big[ \varepsilon^{(N_\delta -1)_+} \sum_{j=1}^N |T_j|\Big]$.
If $\delta \to 0$ then $N_\delta \to N \ge 1$ and thus, using that $m(1)<\infty$ and monotone convergence 
\begin{equation}\label{eq:limit qepsilondelta}
\lim_{\delta \to 0} q_{\eps, \delta } ~=~\E \Big[ \eps^{N-1} \sum_{j=1}^N |T_j| \Big].
\end{equation}
 Moreover,
\begin{align}
q_{\eps, \delta }\E \big[ |B|^{-1}\big]&~=~ \E \Big[ \varepsilon^{(N_\delta -1)_+} \sum_{j=1}^N |T_j||\bar T_j|^{-1} \Big] \nonumber
\\
&=~\E \Big[ N \varepsilon^{(N_\delta -1)_+} \Big] ~\stackrel{\delta \to 0}{\to}~ \E \Big[ N \varepsilon^{N-1} \Big] \label{eq:limit qepsilondelta B}
\end{align}
when $\delta \to 0$, using that $\E\big[ N\big]< \infty$ by assumption. 
Hence, by Eq.s \eqref{eq:limit qepsilondelta} and \eqref{eq:limit qepsilondelta B}, we can choose $\delta $ and $\eps $ small enough such that $q_{\eps , \delta }<1$ and $q_{\eps,\delta} \E\big[|B|^{-1} \big] < 1$. Recall that we assume throughout that $\P(N=0)=0$ to avoid an atom at zero.

From now on, $\delta$ and $\eps$ are fixed and we write $p:=q_{\eps,\delta}<1$. By \eqref{eq:estimate g}, it holds for all $|\xi |\geq t_{\eps }\delta ^{-1}$ that
\begin{equation*}
|g(\xi )|\leq p\E \big[ |g(B\xi )| \big],
\end{equation*} 
and we have that $p \E \big[ |B|^{-1} \big| < 1$.
Recalling that $|g|$ is bounded by $\E \big[ |Z| \big]$, we can apply a Gronwall-type Lemma \cite[Lemma 3.2]{Liu:2001} to the real-valued function $$ g^* : \R_+ \to \R_+, \qquad g^*(t) ~:=~ \max \{ |g(\xi)| \, : \, |\xi|=t\}$$
to conclude that $g^*(t)=O(t^{-1})$. The assertion follows.
%
\end{proof}

\begin{lemma} Suppose $N>0$ a.s., \eqref{a1}-\eqref{a2} with $\alpha \in (1,2)$, \eqref{eq:Z} and \eqref{c1}.
Then  $\partial_{\xi } \phi$ and $\partial_{\bar \xi}$ are square integrable (w.r.t. Lebesgue measure on $\C$).
\end{lemma}

\begin{proof} As before, we focus on $g(\xi )=\partial _{\bar \xi}\phi (\xi )$.
By taking squares in Eq. \eqref{eq:estimate g} and applying Jensen's inequality to the discrete probability measure $\sum_{j=1}^N \frac{1}{N} \delta\{|T_j||g(\bar T_j \xi)|\}$, we obtain
\begin{align}
|g(\xi )|^2 ~&\leq~ \E \Big[ \varepsilon ^{2(N_\delta-1)_+} N^2 \Big(\frac{1}{N}\sum_{j=1}^N |T_j| |g(\bar T_j\xi )| \Big)^2 \Big] \nonumber \\
\label{eq:bound gsquare} &\leq~ \E \Big[ \varepsilon ^{2(N_\delta-1)_+} N \sum_{j=1}^N   \big(|T_j| |g(\bar T_j\xi )|\big)^2  \Big], 
\end{align}
and this estimate is valid for all $\xi$ with $|\xi| \ge t_\epsilon \delta^{-1}$.
Using the decay properties of $g$ provided by Lemma \ref{lem:decay rate g}, we have that the right hand side in \eqref{eq:bound gsquare} is bounded by
$$ \E \Big[ \varepsilon ^{2(N_\delta-1)_+} N \sum_{j=1}^N \big(|T_j| C (1+|T_j||\xi|)^{-1}\big)^2  \Big] ~\le ~ \frac{C}{|\xi|^2} \E \Big[ \varepsilon^{2(N_\delta-1)_+} N^2 \Big],$$
which is finite due to \eqref{c1}.
Defining
$$ I(K) ~:=~ \int_{|\xi| \le K} |g(\xi)|^2 \, d\xi$$
and using the change-of-variables formula (on $\C$), we have with $U:=t_\epsilon \delta^{-1}$
\begin{equation}\label{eq:estimate IK}
I(K)~\le~ I(U) + \E \Big[ \varepsilon^{2(N_\delta-1)_+} N \sum_{j=1}^N I\big( |T_j| K\big) \Big]
\end{equation}
Now choose $\epsilon$ and $\delta$ small such that
$$ \gamma :=\E \Big[ \varepsilon^{2(N_\delta-1)_+} N^2\Big] <1.$$ This is possible since $N_\delta \to N$ a.s. for $\delta \to 0$, $\P(N>1)>0$ and $ \E\big[ N^2\big] < \infty .$ Recall that
$$ \beta:=\E \Big[ \varepsilon^{2(N_\delta-1)_+} N \sum_{j=1}^N \log_+ |T_j|\Big]< \infty $$
by assumption.
The remainder of the proof relies on the following claim.


{\textbf Claim}: For all $m \in \N$, 
$$ I(K) ~\le~\sum_{n=0}^m \gamma^n I(U) + m \gamma^{m-1} \beta C + \gamma^m C \log_+ K,$$
where $C < \infty$ is the constant factor in the growth rate of $g$.

\medskip

If the claim holds, then  $I(K) ~\le~ \frac{I(U)}{1-\gamma} < \infty$ for all $K$, which proves that $g: \C \to \C$ is in $L^2$.

\medskip

{\textbf Proof of the Claim}: We proceed by induction over $m \in \N$. For $N=0$, the estimate on the growth rate of $g$, provided by Lemma \ref{lem:decay rate g}, gives (by possible enlarging $U$)
$$ I(K) \le C \pi + C \log_+ K ~\le~ I(U) + C \log_+K.$$
Note that we are integrating over $\C$, which is a two-dimensional $\R$-space.

Suppose the claim holds for $m \in \N$. This means $I(K) \le a +b \log_+ K$ with the values  $a=\sum_{n=0}^m \gamma^n I(U) + m \gamma^{m-1} \beta C$ , $b=\gamma^m C$. Using Eq. \eqref{eq:estimate IK} to iterate, we obtain
\begin{align*} I(K) ~&\le~ I(U) + \E \Big[ \varepsilon^{2(N_\delta-1)_+} N \sum_{j=1}^N \big( a + b \log_+ |T_j| + b\log_+ K\big) \Big] \\
\nonumber ~&=~ I(U) + \E \Big[ \varepsilon^{2(N_\delta-1)_+} N^2  \big( a + b\log_+ K\big) \Big] + \E \Big[ \varepsilon^{2(N_\delta-1)_+} N \sum_{j=1}^N b \log_+ |T_j| \Big] \\
\label{eq:iterated bound}  ~&=~ I(U) + \gamma a + \gamma b\log_+ K + \beta b \\
~&=~ I(U) +  \gamma \sum_{n=0}^{m} \gamma^n I(U) + m\gamma^m \beta C + \gamma^{m+1} C\log_+ K + \beta \gamma^m C
\end{align*}
which proves the claim.
\end{proof}

Now we are in a position to prove Theorem \ref{thm:continuity of Z}.

\begin{proof}[Proof of Theorem \ref{thm:continuity of Z}.]
Writing $\xi=\xi_1+\imag \xi_2$ and using 
$$ {\partial_{\xi_1} \phi} = \partial_\xi \phi + \partial_{\bar \xi} \phi, \qquad \partial_{\xi_2} \phi = \imag \Big( \partial_\xi \phi - \partial_{\bar \xi} \phi \Big)$$ we have obtained the square integrability of $\partial_{\xi_1} \phi$ and $\partial_{\xi_2} \phi$.

For $j=1,2$, $(\partial_{\xi_j} \phi(\xi)) d\xi$ defines a tempered distribution \cite[VI.2.(4')]{Yosida1980}. By the Plancherel theorem \cite[VI.2.(19)]{Yosida1980}, its Fourier inverse  $$\c F^{-1} \big(\partial_{\xi_j} \phi(\xi) d\xi \big) ~=:~ f_j(z) dz$$ is a tempered distribution defined with square integrable function $f_j$.
On the other hand,
$$ \c F^{-1} \big(\partial_{\xi_j} \phi(\xi) d\xi \big) ~=~ -\imag z_j \c F^{-1} \big(\phi(\xi) d\xi \big) \qquad \Big( z = z_1 + \imag z_2\Big)$$
by \cite[VI.2.(18)]{Yosida1980}. But  $\c F^{-1} \big(\phi(\xi) d\xi \big)$ is nothing but the tempered distribution given by $\P(Z \in dz)$ (in the sense of  \cite[VI.2.(4)]{Yosida1980}), this can be seen as in  \cite[VI.2.(11)]{Yosida1980}. Hence
$$ f_j(z) dz ~=~ - \imag z_j \P(Z \in dz).$$
This shows that for $j=1,2$, $- \imag z_j \P\big(Z \in d(z_1, z_2)\big)$ has a square integrable density $f_j$ on $\C$.
We decompose $\C \setminus \{0\}\simeq \R^2$ into the disjoint union of sets
$$ C_1 = \{ z=(z_1,z_2) \, : \, |z_2| < z_1\}, \ C_2= \{ z \, : \, |z_1| \le z_2, z_2 \neq 0 \}, $$
$C_3 =-C_1$ and $C_4 =-C_2$. On $C_1 \cup C_3$, $\P(Z \in dz)$ has a density given by $(- \imag z_1)^{-1} f_1(z)$, while on $C_2 \cup C_4$, a density for $\P(Z \in dz)$ is given by $(- \imag z_2)^{-1} f_2(z)$. 

{\blue Therefore $\P (Z\in dz )=\P (Z=0)\delta _0+\nu $, where $\nu $ has a density. Then it holds that $\P (Z=0)=\limsup _{|\xi |\to \8}|\phi (\xi )|=0$ in view of Lemma
\ref{lem:decay of phi}  and so 
$\P(Z \in dz)$ is absolutely continuous w.r.t.\ Lebesgue measure on $\C$.}
\end{proof}

\begin{remark}\label{rem:m2 finite}
If $m(2)<\infty$, $\E[N^2] <\infty$ and $\E |Z|^2< \infty$, then   $h(\xi):=\partial ^2_{\bar \xi} \phi(\xi)$ is in $L^{1+\varepsilon}$ for any $\varepsilon >0$, namely $h(\xi) = O(|\xi|^{-2})$. \\
In a similar way, for all $k \in \N$, $k >2$ the following holds: $m(k)<\infty$, $\E[N^k] < \infty$ and $\E |Z|^k < \infty$ imply that $\partial_{\bar \xi}^{(k)} \phi(\xi) = O(|\xi|^{-k})$. {\blue Hence the density $f$ of $\P (Z\in dz)$ belongs to $C^{k-3}(\C \setminus \{ 0\})$ and derivatives of $f$ 
of order for $k-2$ exist in a weak sense on $\C \setminus \{ 0\}$.}
\end{remark}

\begin{proof}[Proof of Remark \ref{rem:m2 finite}]
Firstly, $\E |Z|^2< \infty$ guarantees the existence of $h(\xi)$ and that $h$ is bounded. By the convexity of $m$, the finiteness of $m(0)=\E N < \E [N^2]$ and $m(2)$ yields that $m(1)<\infty$. Hence the assumptions of Lemma \ref{lem:decay rate g} are satisfied and  we obtain the bound $|g(\xi)|=|\partial_{\bar \xi} \phi(\xi)| \le C(1+|\xi|)^{-1}$. Taking derivatives on both sides of Eq. \ref{eq:functional equation for g}, we have
\begin{equation*}\label{eq:functional equation for h}
h(\xi )=\E \Big[\sum_{j=1}^N T_j^2 h(\bar T_j\xi ) \prod_{i\neq j} \phi (\bar T_i\xi )  +  2 \sum_{1 \le i < j \le N} T_i T_j g(\bar T_i\xi ) g(\bar T_j\xi ) \prod_{k\neq i,j} \phi (\bar T_k\xi ) \Big].
\end{equation*}
Using the weaker estimate $|g(\bar{T}_i \xi)| \le C |T_j|^{-1} |\xi|^{-1}$, we deduce
\begin{equation}\label{eq:estimate for h}
|h(\xi )| ~\le~\E \Big[\sum_{j=1}^N |T_j|^2 |h(\bar T_j\xi )| \prod_{i\neq j} |\phi (\bar T_i\xi )| \Big]  +  2C  \E \big[ N^2 \big]  |\xi|^{-2}.
\end{equation}
Now one can proceed as in the proof of Lemma \ref{lem:decay rate g}, defining a complex random variable $B$ such that for any test function $f$
$$ \E f(B) ~=~ p^{-1} \E \Big[ \varepsilon^{(N_\delta -1)_+} \sum_{j=1}^N |T_j|^2 f(\bar T_j)\Big]$$
with the normalization constant $p <1$. Then $p\E \big[ |B|^{-2} \big] \le \E \big[ \varepsilon^{(N_\delta -1)_+} N \big] < 1$ for $\varepsilon$ sufficiently small, and
$$ |h(\xi)| ~\le~ p \E \big[ h(B \xi) \big] + C' |\xi|^{-2}.$$
This is indeed sufficient to proceed as in \cite[Lemma 3.2]{Liu:2001} to conclude that $|h(\xi)|=O(|\xi^{-2})$.

This estimate can then be used in a similar way to produce bounds for $\partial_{\bar \xi }^{(3)} \phi(\xi)$, and so on. 
\end{proof}

\providecommand{\bysame}{\leavevmode\hbox to3em{\hrulefill}\thinspace}
\providecommand{\MR}{\relax\ifhmode\unskip\space\fi MR }
\providecommand{\MRhref}[2]{%
  \href{http://www.ams.org/mathscinet-getitem?mr=#1}{#2}
}
\providecommand{\href}[2]{#2}



{\textbf Acknowledgements:} {We thank Kevin Leckey for helpful discussions during the preparation of the article. The research of S.M.\ was supported by DFG Grant 392119783. {\blue The research of E.D.\ was supported by NCN Grant UMO-2014/15/B/ST1/00060}.  
}



\begin{thebibliography}{10}

\bibitem{Biggins:1992}
J.~D. Biggins, \emph{Uniform convergence of martingales in the branching random
  walk}, Ann. Probab. \textbf{20} (1992), no.~1, 137--151. \MR{1143415}

\bibitem{Biggins.Grey:1979}
John~D. Biggins and D.~R. Grey, \emph{Continuity of limit random variables in
  the branching random walk}, J. Appl. Probab. \textbf{16} (1979), no.~4,
  740--749. \MR{549554 (80j:60107)}

\bibitem{Chauvin+Liu+Pouyanne:2014}
Brigitte Chauvin, Quansheng Liu, and Nicolas Pouyanne, \emph{Limit
  distributions for multitype branching processes of {$m$}-ary search trees},
  Ann. Inst. Henri Poincar\'e Probab. Stat. \textbf{50} (2014), no.~2,
  628--654. \MR{3189087}

\bibitem{Chen.Olvera-Cravioto:2015}
Ningyuan Chen and Mariana Olvera-Cravioto, \emph{Efficient simulation for
  branching linear recursions}, Proceedings of the 2015 Winter Simulation
  Conference (Piscataway, NJ, USA), WSC '15, IEEE Press, 2015, pp.~2716--2727.

\bibitem{Knape+Neininger:2014}
Margarete Knape and Ralph Neininger, \emph{P\'olya urns via the contraction
  method}, Combin. Probab. Comput. \textbf{23} (2014), no.~6, 1148--1186.
  \MR{3265841}

\bibitem{Kolesko.Meiners:2017}
Konrad Kolesko and Matthias Meiners, \emph{Convergence of complex martingales
  in the branching random walk: the boundary}, Electron. Commun. Probab.
  \textbf{22} (2017), 14 pp.

\bibitem{Leckey:2016}
K.~{Leckey}, \emph{{On Densities for Solutions to Stochastic Fixed Point
  Equations}}, ArXiv e-prints (2016).

\bibitem{Liu:2001}
Quansheng Liu, \emph{Asymptotic properties and absolute continuity of laws
  stable by random weighted mean}, Stochastic Process. Appl. \textbf{95}
  (2001), no.~1, 83--107. \MR{1847093 (2002e:60141)}

\bibitem{Meerschaert+Scheffler:2001}
Mark~M. Meerschaert and Hans-Peter Scheffler, \emph{Limit distributions for
  sums of independent random vectors}, Wiley Series in Probability and
  Statistics: Probability and Statistics, John Wiley \& Sons, Inc., New York,
  2001, Heavy tails in theory and practice. \MR{1840531 (2002i:60047)}

\bibitem{Meiners.Mentemeier:2017}
Matthias Meiners and Sebastian Mentemeier, \emph{Solutions to complex smoothing
  equations}, Probab. Theory Related Fields \textbf{168} (2017), no.~1-2,
  199--268. \MR{3651052}

\bibitem{R}
{R Core Team}, \emph{R: A language and environment for statistical computing},
  R Foundation for Statistical Computing, Vienna, Austria, 2015.

\bibitem{Yosida1980}
K{\^o}saku Yosida, \emph{Functional analysis}, sixth ed., Grundlehren der
  Mathematischen Wissenschaften, vol. 123, Springer-Verlag, Berlin, 1980.
  \MR{617913 (82i:46002)}

\end{thebibliography}
\end{document}